\documentclass[12pt]{amsart}

\usepackage{pdfsync}         
\usepackage[active]{srcltx}  

\usepackage{enumerate}
\usepackage{comment}
\usepackage{amssymb}

\usepackage[OT1]{fontenc}    %
\usepackage[usenames]{color}

\usepackage{palatino}
\usepackage[left=3cm,top=3.8cm,right=3cm]{geometry}        

\usepackage{xcolor}
\usepackage[pagebackref]{hyperref}
\hypersetup{
    colorlinks,
    linkcolor={red!60!black},
    citecolor={blue!60!black},
    urlcolor={blue!90!black}
}

\numberwithin{equation}{section}

\theoremstyle{plain}
\newtheorem{theorem}{Theorem}[section]

\newtheorem{prop}[theorem]{Proposition}
\newtheorem{lemma}[theorem]{Lemma}

\theoremstyle{remark}
\newtheorem{remark}[theorem]{Remark}

\theoremstyle{definition}

\newcommand{\e}{\varepsilon}
\newcommand{\N}{\mathbb{N}}
\newcommand{\R}{\mathbb{R}}
\newcommand{\Z}{\mathbb{Z}}
\newcommand{\dist}{\mathrm{dist}}

\newcommand{\cP}{\mathcal{P}}
\newcommand{\cM}{\mathcal{M}}
\newcommand{\cL}{\mathcal{L}}

\newcommand{\cN}{\mathcal{N}}
\newcommand{\cD}{\mathcal{D}}
\newcommand{\cG}{\mathcal{G}}
\newcommand{\cH}{\mathcal{H}}
\newcommand{\cE}{\mathcal{E}}
\newcommand{\cF}{\mathcal{F}}

\DeclareMathOperator{\hdim}{dim_H}
\DeclareMathOperator{\pdim}{dim_P}

\DeclareMathOperator{\ubdim}{\overline{\dim}_B}

\DeclareMathOperator{\supp}{supp}

\newcommand{\wh}{\widehat}
\newcommand{\wt}{\widetilde}

\title{On the Hausdorff dimension of pinned distance sets}

\author{Pablo Shmerkin}
\address{Department of Mathematics and Statistics, Torcuato Di Tella University, and CONICET, Buenos Aires, Argentina}
\email{pshmerkin@utdt.edu}
\urladdr{http://www.utdt.edu/profesores/pshmerkin}
\thanks{P.S. was partially supported by Project PICT 2014-1480 (ANPCyT)}


\begin{document}

\begin{abstract}
We prove that if $A$ is a Borel set in the plane of equal Hausdorff and packing dimension $s>1$, then the set of pinned distances $\{ |x-y|:y\in A\}$ has full Hausdorff dimension for all $x$ outside of a set of Hausdorff dimension $1$ (in particular, for many $x\in A$). This verifies a strong variant of Falconer's distance set conjecture for sets of equal Hausdorff and packing dimension, outside the endpoint $s=1$.
\end{abstract}

\maketitle

\section{Introduction}

\subsection{Background and main result}

Given a set $A\subset\R^d$, with $d\ge 2$, let
\[
\Delta(A)=\{|x-y|:x,y\in A\}
\]
be its distance set. A conjecture of Falconer, dating back to \cite{Falconer85}, states that if $\hdim(A)\ge d/2$, then $\hdim(\Delta(A))=1$, where $\hdim$ stands for Hausdorff dimension. This is open in every dimension $d\ge 2$ (it is easy to see that it fails for $d=1$). From now on, we focus on the planar case, which is the best understood. The current best progress towards Falconer's conjecture in the plane was obtained by Bourgain \cite{Bourgain03} (relying crucially on earlier work of Katz-Tao \cite{KatzTao01}) and by Wolff \cite{Wolff99a}. Assuming only that $\hdim(A)=1$, Bourgain-Katz-Tao proved that $\hdim(\Delta(A))\ge 1/2+\e$, where $\e>0$ is a very small absolute constant (without the $\e$ this bound is due to Falconer \cite{Falconer85}). Under the stronger assumption that $\hdim(A)>4/3$, Wolff proved that $\Delta(A)$ has positive Lebesgue measure.

These results are very deep and have not been improved in well over a decade (see, however, \cite{IosevichLiu17,HambrookIosevichRice17} for recent progress on closely related problems). So, in order to make progress towards the conjecture, one reasonable approach is to try to get stronger conclusions for special classes of sets. The distance set conjecture was established for many dynamically defined sets, such as self-similar and some self-affine sets (\cite{HochmanShmerkin12, Orponen12, FergusonFraserEtAl15}). This is still a very restricted class. Recall that a set $A\subset \R^2$ is called \emph{Ahlfors-David regular} (or AD-regular) with exponent $s$ if
\[
C^{-1} r^s \le \cH^s(A\cap B(x,r)) \le C r^s  \quad\text{for all }x\in A \text{ and } r\in (0,1],
\]
where $\cH^s$ denotes $s$-dimensional Hausdorff measure (in particular, such sets have Hausdorff dimension $s$). Although AD-regular sets are still uniform in terms of size, they allow for far more spatial flexibility than, say, self-similar sets. In \cite{Orponen17}, T. Orponen proved that if $A$ is an AD-regular set of exponent $s\ge 1$, then the \emph{packing} dimension of $\Delta(A)$ is $1$ (for the definition and main properties of packing dimension, see \cite[Chapter 3]{Falconer14}). This falls short of proving Falconer's conjecture for AD-regular sets, since Hausdorff dimension is smaller than packing dimension in general (very roughly speaking, a set has large packing dimension if it is large at infinitely many scales, while in order to have large Hausdorff dimension it must be large at \emph{all} small scales). Further partial progress was achieved in \cite{Shmerkin17}: it was shown there that if $A$ is AD-regular of exponent $s>1$, then the modified lower box-counting dimension of $\Delta(A)=1$. We refer to \cite{Shmerkin17} for the precise definition of  modified lower box-counting dimension, and simply note that it lies between Hausdorff and packing dimension.

A variant of Falconer's distance set problem involves \emph{pinned} distance sets $\Delta_x(A)=\{|x-y|:y\in A\}$. To the best of our knowledge, it may be possible that if $\hdim(A)\ge 1$, then there is $x\in A$ such that $\hdim(\Delta_x A)=1$; of course, this would imply Falconer's conjecture. Peres and Schlag \cite[Corollary 8.4]{PeresSchlag00} proved that if $A\subset\R^2$ is a Borel set with $\hdim(A)>3/2$, then $\Delta_x(A)$ has positive Lebesgue measure outside of a set of $x$ of dimension at most $3/2$ (in particular, for most $x\in A$). Again, one can try to get better results for special classes of sets. In \cite[Corollaries 1.2 and 1.5]{Shmerkin17}, several results were established on the box dimensions of pinned distance sets of AD-regular sets of dimension $>1$, and this left open the question of whether box dimension could be replaced with Hausdorff dimension.

Let $\pdim$ denote packing dimension. If $A$ is AD-regular of exponent $s$, then both the Hausdorff and packing dimensions of $A$ equal $s$. On the other hand, there are many sets for which $\hdim(A)=\pdim(A)$ which are far from AD-regular. To emphasize the qualitative difference between both classes, let $S$ denote the set of natural numbers $n$ such that $(2k)^2 \le n<(2k+1)^2$ for some $k\in\N$, and let
\[
A = \left\{ \sum_{n=1}^\infty a_n 2^{-n}: a_n\in \{0,1\}, a_n=0 \text{ for } n\in S \right\}.
\]
Then $\hdim(A)=\pdim(A)=1/2$ (this is standard), but $A$ is not AD-regular. Indeed, one can check that the arithmetic sums $A+\cdots+A$ all have (Hausdorff and packing) dimension $1/2$, while it is known that $\hdim(A+A)>\hdim(A)$ if $A$ is AD-regular. This latter fact follows from the additive part of Bourgain's sum-product theory \cite{Bourgain03, Bourgain10} that established the $1/2+\e$ bound for distance sets. Similar examples can be constructed in the plane.

Here we prove that if $A\subset\R^2$ has (equal) Hausdorff and packing dimension $s>1$, then $\Delta_x(A)$ has full Hausdorff dimension for all $x$ outside of a set of Hausdorff dimension $1$:
\begin{theorem} \label{thm:distance-set}
Let $A\subset\R^2$ be a Borel set with $\hdim(A)=\pdim(A)=s>1$. Then
\[
 \hdim\{ x\in\R^2: \hdim(\Delta_x(A))<1  \} \le 1.
\]
In particular, there exists $x\in A$ such that $\hdim(\Delta_x(A))=1$.
\end{theorem}

This verifies the strong form of Falconer's distance set conjecture, outside of the endpoint $s=1$, for this class of sets. Of course, it is enough to assume that $A$ contains a set of Hausdorff and packing dimension $s>1$. With this observation, this theorem recovers, generalizes and improves the results on distance sets from \cite{Orponen12, FergusonFraserEtAl15, Orponen17, Shmerkin17}, again outside of the endpoint. We also point out that packing dimension is smaller than upper box counting (Minkowski) dimension, so equality of Hausdorff and box dimension also implies the conclusion of the theorem.

\subsection{About the proof}

We make some remarks about the proof and, particularly, how it relates to previous approaches in the literature. In many recent papers about (linear or non-linear) projections, such as \cite{PeresShmerkin09, HochmanShmerkin12, Orponen12, FergusonFraserEtAl15, Orponen17, Shmerkin17}, the strategy is to apply some version of Marstrand's Projection Theorem to a multi-scale decomposition of a suitable measure $\mu$ supported on $A$ (the set being projected). While the details vary, in all of the cited papers the scales $m_j$ in the multi-scale decomposition form an arithmetic progression, and this is central to the respective methods. Although we still use projection theorems on a multi-scale decomposition of $\mu$, one of the key features of this work is that we are able to make $m_j$ rapidly increasing, and this is crucial for obtaining results for Hausdorff dimension instead of box or packing dimension, and also for widening the class of AD-regular sets to sets with equal Hausdorff and packing dimension. More precisely, we will take $m_j=2^{\lfloor (1+\e)^j\rfloor}$ for a small but fixed $\e>0$. These ``hyperdyadic'' scales were also used to study distance sets in \cite{KatzTao01}.

In the special case of pinned distance projections $y\mapsto  |x-y|$, another key ingredient is the study of directions. Given $x\neq y\in\R^2$, let $\theta(x,y)=x-y/|x-y|\in S^1$ denote the direction determined by $x$ and $y$. Since the gradient of $y\mapsto |x-y|$ is precisely $\theta(x,y)$  one needs to know that, in some sense, $A$ determines many directions (or, for pinned distance sets $\Delta_x(A)$, that there are many directions between $x$ and points in $A$). To be more precise, one needs to know that $\theta(x,y)$ is often a ``good'' direction in the sense of a projection theorem, for the restriction of $\mu$ to a small cube containing $y$. In the study of distance sets, these sets of good directions are always large, but otherwise unknown. The strategy of \cite{Orponen12, Orponen17}, which is also used in \cite[Theorem 1.1]{Shmerkin17}, is to employ the fact that if $A$ has Hausdorff dimension $>1$ (or, more generally, is purely unrectifiable) then for most points $y\in A$, the set of directions $\{\theta(x,y):x\in A\}$ is dense. This is not quantitative enough to allow the scales $m_j$ to grow quickly. In \cite{MattilaOrponen16}, Mattila and Orponen proved (among other things) that given a set $A$ with $\hdim(A)>1$, the image of $A$ under the radial projection $x\mapsto \theta(x,y)$ has positive Lebesgue measure, for all $y$ outside of a set of Hausdorff dimension $1$. Inspired by this, we establish in Section \ref{sec:directions} a similar but more quantitative result. This enables us to find a suitable ``vantage point'' $x$ for which we can prove that the pinned distance set is large; see Theorem \ref{thm:good-base-point}.

A problem that many of the previously cited papers have to contend with is that Marstrand-type projection theorems are intimately related to energies, but energies do not have a nice multi-scale decomposition, while entropy does. Also, energy is much more sensitive to ``bad'' pieces of the measure with small mass. The main role of the assumption of equal Hausdorff and packing dimension is to provide a good multi-scale decomposition of energy, and hence a near-optimal projection theorem at each scale. This allows us to carry more of the argument with energies and $L^2$ norms instead of entropies (although entropy still plays a major role). In particular, we exploit the trivial but important fact that $L^2$ norms and energies are monotone, in the sense that restricting a measure to a subset cannot increase these quantities.

\section{Notation}
\label{sec:notation}

We set up some notation, especially concerning measures.

By $A\lesssim B$ we mean $0\le A\le CB$ for some constant $C>0$; if $C$ depends on some parameter, this is sometimes denoted by a subscript, i.e. $A\lesssim_\e B$ means $0\le A\le C(\e)B$. We write $A\gtrsim B$ for $B\lesssim A$, and $A\approx B$ for $A\lesssim B\lesssim A$, again with the possibility of using subscripts to make dependencies explicit.

We denote the family of all Borel probability (resp. Radon) measures on a metric space $X$ by $\cP(X)$ (resp. $\cM(X)$). If $f:X\to Y$ and $\mu\in\cM(X)$, the \emph{push-forward measure} $f\mu$ is defined as $f\mu(A)=\mu(f^{-1}A)$ for all measurable $A$. (This is often denoted by $f_\#\mu$.)

If $\mu\in\cM(X)$ and $\mu(A)>0$, then $\mu|_A$ is the restriction of $\mu$ to $A$ and, provided also $\mu(A)<\infty$, we also denote the normalized restriction by $\mu_A=\tfrac{1}{\mu(A)}\mu|_A$.

We work in an ambient dimension $d$; this will always be $1$ or $2$ in this paper.  We denote by $\cD_k^{(d)}$ the partition of $\R^d$ into half-open dyadic cubes
\[
\left\{  [j_1 2^{-k}, (j_1+1)2^{-k})\times \cdots\times [j_d 2^{-k}, (j_d+1)2^{-k}): j_1,\ldots,j_d\in\Z \right\}.
\]
We drop the superindex $(d)$ when it is clear from context.

Given $\mu\in\cP(\R^d)$ and $m\in\N$, we write
\[
\mu^{(m)} = \sum_{Q\in\cD_m} \mu(Q) \cL_Q,
\]
where $\cL$ denotes Lebesgue measure on $\R^d$. That is, $\mu^{(m)}$ is a discretized version of $\mu$ at scale $2^{-m}$. When $Q \in \cD_m$, we also denote
\[
\mu^Q = T_Q\mu_Q,
\]
where $T_Q$ is the homothety renormalizing $Q$ back to $[0,1)^d$. We sometimes shorten $(\mu^Q)^{(m)}$ to $\mu^{Q,(m)}$ for convenience.

If $\mu\in\cM(\R^d)$ is absolutely continuous, we denote its density also by $\mu$.

We have already introduced the notation $\hdim,\pdim$ for Hausdorff and packing dimension in the introduction. We will denote upper box counting (or Minkowski) dimension by $\ubdim$. We refer to \cite[Chapters 2 and 3]{Falconer14} for the definitions and basic properties of these notions of dimension.

\section{Preliminaries: energy, entropy, and projections}
\label{sec:preliminaries}

In this section we present some preliminary results related to energies, entropy, and their behaviour under projections. Most of the material is standard, but in order to make the paper self-contained, we give complete proofs of most statements.

On the cube $[0,1)^d$ we consider the dyadic metric: $d(x,y)=2^{-|x\wedge y|}$, where $|x\wedge y|=\max\{ j: y\in D_j(x)\}$. We denote the $s$-energy of a measure $\mu$ on $[0,1)^d$ with respect to this metric by
\[
\cE_s(\mu) = \iint 2^{s|x\wedge y|} \,d\mu(x)d\mu(y).
\]
The following well-known representation of $\cE_s$ will be useful.
\begin{lemma} \label{lem:energy-to-correlation} If $\mu\in\cP([0,1)^d)$ and $s\in (0,d)$, then
\[
\cE_s(\mu) \approx_s \sum_{j=1}^\infty 2^{sj} \sum_{Q\in\cD_j} \mu(Q)^2.
\]
\end{lemma}
\begin{proof}
We compute:
\begin{align*}
\cE_s(\mu) &= \iint \left(1+\sum_{j=1}^{|x\wedge y|} 2^{s j}-2^{s(j-1)}\right) \,d\mu(x)d\mu(y)\\
&= 1 +\iint \sum_{j=1}^\infty \sum_{Q\in\cD_j} (2^{sj}-2^{s(j-1)}) \mathbf{1}_Q(x)\mathbf{1}_Q(y)\, d\mu(x)d\mu(y)\\
&= 1 + \sum_{j=1}^\infty (1-2^{-s})2^{sj} \sum_{Q\in\cD_j} \mu(Q)^2.
\end{align*}
\end{proof}
As a first application, we have:
\begin{lemma} \label{lem:energy-of-discretization}
Let $0<s<d$. Then
\[
\cE_s(\mu^{(m)}) \approx_{d,s} \sum_{j=1}^m 2^{sj} \sum_{Q\in\cD_j} \mu(Q)^2.
\]
\end{lemma}
\begin{proof}
Simply note that if $Q'\subset Q$ with $Q'\in\cD_{m'}$ and $Q\in\cD_m$, then $\mu^{(m)}(Q')=2^{d(m-m')}\mu(Q)$, so that
\[
\sum_{Q'\in\cD_{m'}} \mu(Q')^2 = 2^{(s-d)(m'-m)} \left( 2^{sm} \sum_{Q\in\cD_m} \mu(Q)^2\right).
\]
In light of Lemma \ref{lem:energy-to-correlation}, adding up over $m'\ge m$ finishes the proof.
\end{proof}

As a direct consequence of Lemma  \ref{lem:energy-of-discretization}, we have the following multi-scale decomposition of energy.
\begin{lemma} \label{lem:multiscale-energy}
Let $\mu\in\cP([0,1)^d)$. Then for every sequence $0=m_0<m_1<\ldots<m_k$,
\[
\cE_s(\mu^{(m_k)}) \approx_{d,s} \sum_{j=0}^{k-1} 2^{s m_j} \sum_{Q\in\cD_{m_j}} \mu(Q)^2 \cE_s(\mu^{Q,(m_{j+1}-m_j)})
\]
\end{lemma}
\begin{proof}
It follows from Lemma \ref{lem:energy-of-discretization} and the definitions that
\[
 \sum_{Q\in\cD_{m_j}}   \mu(Q)^2 \cE_s(\mu^{Q,(m_{j+1}-m_j)}) \approx_{d,s} \sum_{k=m_j+1}^{m_{j+1}} 2^{s(k-m_j)}\sum_{Q'\in \cD_k} \mu(Q')^2.
\]
Now just add over $j$.
\end{proof}

Note that this continues to hold if $m_j$ is non-decreasing and $m_{j+1}=m_j$ for $\lesssim 1$ values of $j$.

While for us it is more convenient to work with the dyadic version of energy, the relationship between the energy of a measure and that of its projections is classically stated for Euclidean energy, defined as
\[
\cE_s^*(\mu) =\iint |x-y|^{-s} \,d\mu(x)d\mu(y).
\]
Thankfully, the following result of Pemantle and Peres \cite{PemantlePeres95} asserts that both kinds of energy are comparable up to a constant depending only on the ambient dimension.
\begin{theorem}[{\cite[Theorem 3.1]{PemantlePeres95}}]  \label{thm:dyadic-to-euclidean-energy}
\[
\cE_s(\mu) \approx_d \cE_s^*(\mu).
\]
\end{theorem}

Recall that the $(2,\gamma)$-Sobolev norm of a probability measure is defined as
\[
\|\nu\|_{2,\gamma}^2 = \int |\xi|^{2\gamma} |\wh{\nu}(\xi)|^2 d\xi,
\]
where $\wh{\nu}(\xi) = \int e^{2\pi i  \xi\cdot x}d\nu(x)$ is the Fourier transform of $\nu$.  We will often implicitly use that $\|\nu\|_2^2 \lesssim \|\nu\|_{2,\gamma}^2$ for positive $\gamma$. The following version of Marstrand's projection theorem will be one of the key tools in the proof of Theorem \ref{thm:distance-set}. Given $\mu\in\cP(\R^2)$ and $\theta\in S^1$, we denote by $\mu_\theta$ the push-down of $\mu$ under the orthogonal projection $\Pi_\theta(x)= \theta \cdot x$.
\begin{theorem} \label{thm:integral-sobolev-energy}
Let $\mu$ be a probability measure on $\R^2$ and let $\gamma\in (-1/2,1/2)$. Then
\[
\int_{S^1} \|\mu_\theta\|_{2,\gamma}^2  \,d \sigma(\theta) \approx  \cE_{1+2\gamma}(\mu),
\]
where $\sigma$ is normalized Lebesgue measure on the circle.
\end{theorem}
\begin{proof}
See \cite[Theorem 4.5]{Mattila04} for the proof with $\cE^*$ in place of $\cE$. Thanks to Theorem \ref{thm:dyadic-to-euclidean-energy}, the statement also holds for $\cE$.
\end{proof}

While on $\R^2$ we will find it more convenient to work with energies, in order to  deduce that the pinned distance sets are large we will estimate their sizes via entropy. We begin by recalling some definitions related to the latter. We denote Shannon entropy of the probability measure $\mu$ with respect to a finite measurable partition $\cF$ (of $\supp(\mu)$) by $H(\mu,\cF)$, and the conditional entropy with respect to the finite measurable partition $\mathcal{G}$ by $H(\mu,\cF|\mathcal{G})$. These are defined as
\begin{align*}
H(\mu,\cF) &= \sum_{F\in\cF} -\mu(F)\log\mu(F),\\
H(\mu,\cF|\mathcal{G}) &= \sum_{G\in\mathcal{G}:\mu(G)>0} \mu(G) H(\mu_G,\cF).
\end{align*}
We follow the usual convention $0\cdot \log(0)=0$. Further, if $\mu\in\cP(\R^d)$ has bounded support we denote by $H_k(\mu)$ the normalized entropy $\tfrac{H(\mu,\cD_k)}{k}$, and note that if $\mu\in\mathcal{P}([0,1)^d)$, then $0\le H_k(\mu)\le d$. This is a particular case of the general fact that
\begin{equation} \label{eq:upper-bound-entropy-size-partition}
H(\mu,\cF) \le \log|\cF|.
\end{equation}

We record the following immediate consequence of the concavity of the logarithm.
\begin{lemma} \label{lem:L2-to-entropy}
If $\mu\in\cP(\R^d)$, then
\[
H_m(\mu) \ge  d- \frac{1}{m}\log\|\mu^{(m)}\|_2^2.
\]
\end{lemma}
\begin{proof}
First, note that
\[
\|\mu^{(m)}\|_2^2 = 2^{dm} \sum_{Q\in\cD_m} \mu(Q)^2.
\]
Since
\[
-H(\mu,\cD_m) = \sum_{Q\in\cD_m} \mu(Q)\log\mu(Q) \le \log\left(\sum_{Q\in\cD_m} \mu(Q)^2\right)
\]
by the concavity of the logarithm, the claim follows.
\end{proof}
The previous lemma will be used in conjunction with the following one, asserting that if one first projects a measure and then discretizes, the $L^2$ norm is roughly the same as if one first discretizes, and then projects.
\begin{lemma} \label{lem:discretize-project}
\[
\|(\mu_\theta)^{(m)}\|_2^2 \approx \|\Pi_\theta(\mu^{(m)})\|_2^2.
\]
\end{lemma}
\begin{proof}
The densities of  $(\mu_\theta)^{(m)}$,  $\Pi_\theta(\mu^{(m)})$ will be denoted  by $f,g$ respectively.

We first show that $\|f\|_2^2 \lesssim \|g\|_2^2$. Let $I\in\cD_m^{(1)}$, and denote the interval with the same center as $I$ and five times the length by $5I$. Since every $Q\in\cD_m^{(2)}$ such that $\Pi_\theta(Q)\cap I\neq\varnothing$ satisfies $\Pi_\theta(Q)\subset 5I$, we have that.
\[
(\mu_\theta)^{(m)}(I) \le \Pi_\theta(\mu^{(m)})(5I).
\]
Therefore
\begin{align*}
\int_I f^2 &= 2^m \left((\mu_\theta)^{(m)}(I)\right)^2 \le 2^m \left(\Pi_\theta(\mu^{(m)})(5I)\right)^2 \lesssim \int_{5I} g^2,
\end{align*}
using Cauchy-Schwartz for the last inequality. Adding over all $I\in\cD_m^{(1)}$ yields the claim.

For the opposite inequality, note that if $x\in I\in\cD_m$ then (again using that if $\Pi_\theta(Q)\cap I\neq\varnothing$, then $\Pi_\theta(Q)\subset 5I$)
\[
g(x) = 2^{2m} \sum_{Q\in\cD_m} \mu(Q) \mathcal{H}^1(Q\cap \Pi_\theta^{-1}(x)) \lesssim 2^m \mu_\theta(5I) = 2^m (\mu_\theta)^{(m)}(5I).
\]
Hence, using $(\sum_{i=1}^5 a_i)^2 \lesssim \sum_{i=1}^5 a_i^2$,
\[
\int_I g^2 \lesssim 2^m \sum_{J\in\cN(I)} \mu_\theta(J)^2,
\]
where $\cN(I)$ are the five dyadic intervals making up $5 I$. Adding up over $I\in\cD_m$ yields the claim.
\end{proof}

The next proposition is the key device that will allow us to bound from below the (normalized) entropy of pinned distance measures in terms of a multi-scale formula involving localized entropies. A local variant  of this goes back to \cite{HochmanShmerkin12}, while the relationship between local and global entropy is explored in \cite{Hochman14}. The particular version below is a small adaptation of results from \cite{Orponen17}. We note that, although the proof follows, with minor changes, by combining those of \cite[Lemma 3.5, Remark 3.6 and Lemma 3.12]{Orponen17}, there is a conceptual difference with all the cited works, already remarked in the introduction: in all of them, the sequence $m_j$ forms an arithmetic progression (and this was essential for the methods in those papers), while for us it will be crucial that $m_{j+1}-m_j\to\infty$ at a sufficiently fast rate.
\begin{prop} \label{prop:entropy-of-pinned-dist-measures}
Let $\mu\in\cP([0,1)^d)$, and let $y\in \R^d\setminus \supp(\mu)$. Let $0=m_0\le m_1\le \ldots \le m_k$, and write $d_j=m_{j+1}-m_j$. Suppose $d_j \le m_j+1$ for all $j$. Then
\[
H(\Delta_y\mu,\cD_{m_k}) \ge -C k +\sum_{j=0}^{k-1} \sum_{Q\in\cD_{m_j}} \mu(Q) H\left(\mu^Q_{\theta(y,x_Q)} ,\cD_{d_j}\right),
\]
where $x_Q$ are arbitrary points in $Q$, and $C>0$ depends only on $\dist(y,\supp(\mu))$.
\end{prop}

In the proof we will require some further elementary properties of entropy:
\begin{enumerate}[(A)]
\item \label{enum:equivalent-partitions}  If $\cF, \cG$ have the property that each element of $\cF$ hits at most $N$ elements of $\cG$ and vice-versa, then
\[
|H(\mu,\cF)-H(\mu,\cG)| \le \log N.
\]
\item \label{enum:refining-partitions} If $\cG$ refines $\cF$ (that is, each element of $\cF$ is a union of elements in $\cG$), then
\[
H(\mu,\cF|\cG) = H(\mu,\cG)-H(\mu,\cF).
\]
\item \label{enum:concavity-of-entropy} Conditional entropy is concave as a function of the measure: for $t\in [0,1]$,
\[
H(t\mu+(1-t)\nu,\cF|\cG) \ge t H(\mu,\cF|\cG)+(1-t)H(\nu,\cF|\cG).
\]
\end{enumerate}

The proof the proposition depends on a linearization argument, which we present first. It is very similar to \cite[Lemma 3.12]{Orponen17}.
\begin{lemma} \label{lem:linearization}
Under the assumptions of Proposition \ref{prop:entropy-of-pinned-dist-measures}, if $Q\in\cD_{m_j}$ has positive $\mu$-measure, then
\begin{equation} \label{eq:linear-to-nonlinear-entropy}
\left| H\big(\Delta_y(\mu_Q) ,\cD_{m_{j+1}}|\cD_{m_j}\big) - H\big(\Pi_{\theta(y,x_Q)}(\mu_Q),\cD_{m_{j+1}}|\cD_{m_j}\big) \right|\lesssim 1 ,
\end{equation}
with the implicit constant depending on $\dist(y,\supp(\mu))$ only.
\end{lemma}
\begin{proof}
Since $\Delta_y(Q)$ and $\Pi_{\theta(y,x_Q)}(Q)$ have diameter $\lesssim 2^{-m_j}$, it follows that if $f=\Delta_y$ or $\Pi_{\theta(y,x_Q)}$, then $H(f\mu_Q,\cD_{m_j})\lesssim 1$ and so, thanks to property \eqref{enum:refining-partitions} above, it is enough to prove that
\[
\left| H\big(\Delta_y(\mu_Q) ,\cD_{m_{j+1}}\big) - H\big(\Pi_{\theta(y,x_Q)}(\mu_Q),\cD_{m_{j+1}}\big) \right|\lesssim 1.
\]
This is equivalent to
\begin{equation} \label{eq:equivalent-partitions}
\left| H(\mu_Q ,\mathcal{F}) - H(\mu_Q,\mathcal{G}) \right|\lesssim 1,
\end{equation}
where
\begin{align*}
\mathcal{F} &= \{ \Delta_y^{-1}(I): I \in \cD_{m_{j+1}}, \Delta_y^{-1}(I)\cap Q\neq\varnothing\},\\
\mathcal{G} &= \{ \Pi_{\theta(y,x_Q)}^{-1}(I): I \in \cD_{m_{j+1}}, \Pi_{\theta(y,x_Q)}^{-1}(I)\cap Q\neq\varnothing\}.
\end{align*}

Now, note that if $z_1,z_2\in Q$ then, since $\nabla \Delta_y(z) = \theta(z,y)$, there is $z_3$ in the segment joining them such that
\begin{align}
|\Delta_y(z_1)-\Delta_y(z_2)| &= |\Pi_{\theta(y,z_3)}(z_1-z_2)| \nonumber \\
&\le |\Pi_{\theta(y,x_Q)}(z_1-z_2)| + \|\Pi_{\theta(y,z_3)}-\Pi_{\theta(y,x_Q)}\||z_1-z_2| \nonumber \\
&\lesssim |\Pi_{\theta(y,x_Q)}(z_1-z_2)| + 2^{-m_j} 2^{-m_j} \nonumber \\
&\le |\Pi_{\theta(y,x_Q)}(z_1)-\Pi_{\theta(y,x_Q)}(z_2)| + 2\cdot 2^{-m_{j+1}}. \label{eq:linearization}
\end{align}
In the third line we used that $|\theta(y,z_3)-\theta(y,x_Q)|\lesssim 2^{-m_j}$ (here the constant depends on the distance from $y$ to $\supp(\mu)$), while in the last line we used the hypothesis $d_j \le m_j+1$. This shows that each element of $\mathcal{G}$ intersects $\lesssim 1$ elements of $\mathcal{F}$. This also holds with the partitions interchanged, with the same argument, so property \eqref{enum:equivalent-partitions} above yields that \eqref{eq:equivalent-partitions} is verified and, with it, the lemma.
\end{proof}

\begin{proof}[Proof of Proposition \ref{prop:entropy-of-pinned-dist-measures}]
We estimate:
\begin{align*}
H(\Delta_y\mu,\cD_{m_k}) &= \sum_{j=0}^{k-1}  H(\Delta_y\mu,\cD_{m_{j+1}}|\cD_{m_j}) \\
&= \sum_{j=0}^{k-1} H\left(\sum_{Q\in\cD_{m_j}} \mu(Q) \Delta_y(\mu_Q) ,\cD_{m_{j+1}}|\cD_{m_j}\right)\\
&\ge \sum_{j=0}^{k-1} \sum_{Q\in\cD_{m_j}} \mu(Q) H(\Delta_y(\mu_Q) ,\cD_{m_{j+1}}|\cD_{m_j})\\
&\ge \sum_{j=0}^{k-1} \sum_{Q\in\cD_{m_j}} \mu(Q) \left(H((\mu_Q)_{\theta(y,x_Q)},\cD_{m_{j+1}}|\cD_{m_j}) - C\right)\\
&= \sum_{j=0}^{k-1} \sum_{Q\in\cD_{m_j}} \mu(Q) (H(\mu^Q_{\theta(y,x_Q)} ,\cD_{d_j}|\cD_0) -C) \\
&= \sum_{j=0}^{k-1} \sum_{Q\in\cD_{m_j}} \mu(Q) (H(\mu^Q_{\theta(y,x_Q)} ,\cD_{d_j}) - C).
\end{align*}
We used property \eqref{enum:refining-partitions} in the first line, and property \eqref{enum:concavity-of-entropy} in the third line. In the fourth line we invoked  Lemma \ref{lem:linearization},  and in the fifth we appealed to the fact that $\Pi_{\theta(y,x_Q)}$, being linear, commutes with scalings.
\end{proof}

\begin{remark}
If in Proposition \ref{prop:entropy-of-pinned-dist-measures} we assume that $d_j \le m_j + T$, for some constant $T\ge 1$, then the same conclusion holds, except that the constant $C$ will now depend also on $T$. Indeed, the assumption $d_j \le m_j+1$ is used only in \eqref{eq:linearization} in the proof of Lemma \ref{lem:linearization}. If instead we only have $d_j \le m_j+T$, then the factor $2$ in \eqref{eq:linearization} has to be replaced by $2^T$ but in the rest of the argument this only affects the value of the constant $C$.
\end{remark}

\section{Directions determined by two measures, and good vantage points for multi-scale projections}
\label{sec:directions}

\subsection{A quantitative circular projection theorem}

We begin by reviewing the main properties of conditional measures on lines. Let $\nu\in\cP([0,1]^2)$. Given $x\in\R$ and $\theta\in S^1$, we define
\[
\nu_{\theta,x} = \lim_{r\downarrow 0} \frac{1}{2r} \nu|_{T(\theta,x,r)},
\]
if the limit exists, where the limit denotes weak convergence, and $T(\theta,x,r)$ is the tube of width $2r$ around the line $\Pi_\theta^{-1}(x)$. If $y\in \R^2$, we also write $\nu_{\theta,y}=\nu_{\theta,\Pi_\theta y}$, again if the latter exists.

The measures  $\nu_{\theta,x}$ are known as \emph{conditional} or \emph{sliced} measures. If $\cE_1(\nu)<+\infty$, then for  almost all $\theta$ the projected measure $\nu_\theta$ is absolutely continuous (this follows e.g. from Theorem \ref{thm:integral-sobolev-energy}), the conditional measures exist for (Lebesgue) almost all $x$, and the disintegration formula
\[
\nu =\int \nu_{\theta,x} \,dx
\]
holds. In particular, $\nu_{\theta,x}(\R^2)=\nu_\theta(x)$ for almost all $(\theta,x)$.  See \cite[Chapter 10]{Mattila95} for more details.

The following theorem is inspired by the results of Mattila and Orponen in \cite{MattilaOrponen16} concerning spherical projections (in particular \cite[Theorem 5.2]{MattilaOrponen16}, although we do not pay attention to the dimension of the intersections). We need, however, a quantitative formulation, and this requires us to deal with Sobolev norms. We present a planar version, but it can easily be extended to arbitrary dimensions.

\begin{theorem} \label{thm:large-direction-set}
There is an absolute $\tau>0$ such that the following holds. Let $\mu,\nu$ be probability measures on $[0,1]^2$ with $\cE_s(\mu),\cE_s(\nu)<\infty$ for some $s\in (1,2)$. Then
\[
(\mu\times\sigma)\{ (y,\theta):  |\nu_{\theta,y}|\ge\tau \} \gtrsim \left(\cE_s(\mu)\cE_s(\nu)\right)^{-\frac{1}{(s-1)}}.
\]
Here $|\nu_{\theta,y}|$ is the total mass of the conditional measure $\nu_{\theta,y}$, and the statement $|\nu_{\theta,y}|\ge \tau$ should be understood as implying in particular the existence of $\nu_{\theta,y}$.
\end{theorem}
\begin{proof}
In the course of the proof, $C_i$ denote positive absolute constants. To begin, recall that $\mu_\theta$ and $\nu_\theta$ are absolutely continuous (with $L^2$ density) for almost all $\theta$; this follows e.g. from Theorem \ref{thm:integral-sobolev-energy}. Using Plancherel's formula and the expression of the (mutual) energy in terms of the Fourier transform, it was shown in \cite[Eq. (3.1)]{MattilaOrponen16} that
\begin{equation} \label{eq:mutual-energy-large}
\iint \mu_\theta(x)\nu_\theta(x) \,dx\,d\sigma(\theta) = C_1 \int |x-y|^{-1}\,d\mu(x)\,d\nu(y) \ge C_2,
\end{equation}
where the inequality holds since $\mu,\nu$ are probability measures on $[0,1]^2$.  On the other hand, letting $1+2\gamma=s$,
\begin{equation}
\begin{split}
\iint_{|\xi|\ge K}| \wh{\mu}_\theta(\xi)\wh{\nu}_\theta(\xi)| d\xi\,d\sigma(\theta) &\le K^{-2\gamma}  \iint |\xi|^{\gamma}  |\wh{\mu}_\theta(\xi)||\xi|^{\gamma}|\wh{\nu}_\theta(\xi) |d\xi\,d\sigma(\theta)  \\
\label{eq:large-frequencies-decay} &\le K^{-2\gamma} \int \|\mu_\theta\|_{2,\gamma} \|\nu_\theta\|_{2,\gamma} \,d\sigma(\theta) \\
&\le K^{-2\gamma} \left(\int \|\mu_\theta\|_{2,\gamma}^2 \,d\sigma(\theta)\right)^{1/2} \left(\int \|\nu_\theta\|_{2,\gamma}^2 \,d\sigma(\theta)\right)^{1/2}  \\
&\lesssim K^{1-s} \cE_s(\mu)^{1/2} \cE_s(\nu)^{1/2}.
\end{split}
\end{equation}
We applied Cauchy-Schwartz in the second and third line, and Theorem \ref{thm:integral-sobolev-energy} in the last line. (Here and in the sequel we appeal to Theorem \ref{thm:dyadic-to-euclidean-energy} to pass between dyadic and Euclidean energies.) Let $\phi:\R\to\R$ be Schwartz function such that $\wh{\phi}:\R\to\R$  and $\wh{\phi}|_{[-1,1]}\equiv 1$, and write $\phi_K(x)=K\phi(K x)$. We deduce from \eqref{eq:mutual-energy-large}, \eqref{eq:large-frequencies-decay} and Plancherel that
\begin{equation}
\begin{split} \label{eq:mutual-energy-approx-large}
\iint \mu_\theta(x) \phi_K*\nu_\theta(x) \,dx\,d\sigma(\theta) &= \iint \wh{\phi_K}(\xi) \wh{\mu}_\theta(\xi) \overline{\wh{\nu}_\theta(\xi)}\,d\xi\,d\sigma(\theta) \\
&\ge C_2 - C_3 K^{1-s} \cE_s(\mu)^{1/2} \cE_s(\nu)^{1/2}.
\end{split}
\end{equation}
Likewise, applying \eqref{eq:large-frequencies-decay} with $\mu=\nu$ and Plancherel, we get
\begin{equation*}
\int \| \nu_\theta - \phi_K*\nu_\theta \|_2^2 \,d\sigma(\theta) \lesssim K^{1-s}\cE_s(\nu).
\end{equation*}
Given $\tau\in (0,1)$ (to be chosen momentarily), pick $K$ such that
\[
\tau=K^{(1-s)/2} \cE_s(\mu)^{1/2} \cE_s(\nu)^{1/2}.
\]
Applying Cauchy-Schwartz twice and Theorem \ref{thm:integral-sobolev-energy}, we further estimate
\begin{align*}
\iint_{ \nu_\theta(x)\le \tau} \mu_\theta(x)|\phi_K*\nu_\theta(x)| &\, dx\, d\sigma(\theta) \le \iint  \mu_\theta(x) (|\nu_\theta(x)-\phi_K*\nu_\theta(x)|+\tau)\,dx\,d\sigma(\theta)\\
&\le \tau+\int \|\mu_\theta\|_2\|\nu_\theta-\phi_K*\nu_\theta\|_2  \,d\sigma(\theta) \\
&\le \tau+\left(\int  \|\mu_\theta\|_2^2\,d\sigma(\theta)\right)^{1/2}\left(\int \|\nu_\theta-\phi_K*\nu_\theta\|_2^2 \,d\sigma(\theta)\right)^{1/2} \\
&\le \tau+C_4 K^{(1-s)/2} \cE_s(\mu)^{1/2} \cE_s(\nu)^{1/2}\\
&\lesssim \tau.
\end{align*}
Combining this with  \eqref{eq:mutual-energy-approx-large} and using that $K^{1-s}\cE_s(\mu)^{1/2}\cE_s(\nu)^{1/2} \lesssim \tau^2\le \tau$, we deduce that
\[
\iint_{\nu_\theta(x)>\tau} \mu_\theta(x)\phi_K*\nu_\theta(x) \, dx\, d\sigma(\theta)\ge C_2- C_5\tau.
\]
We now fix $\tau$ (and hence $K$) such that $C_2-C_5\tau=C_2/2$, i.e.
\[
 K = C_6 \left(\cE_s(\mu)\cE_s(\nu)\right)^{\frac{1}{(s-1)}}.
\]
Since $\|\phi_K*\nu_\theta\|_\infty \le K\|\phi\|_\infty \lesssim K$, we deduce that
\begin{align*}
\frac{C_2}{2} &\le \iint_{\nu_\theta(x)\ge\tau} \mu_\theta(x)|\phi_K*\nu_\theta(x)|\,dx\,d\sigma(\theta) \\
&\lesssim K \int \mu_\theta\{ x: \nu_\theta(x)\ge\tau \}\,d\sigma(\theta)\\
&= K \int \mu_\theta\{ x: |\nu_{\theta,x}| \ge \tau \} \,d\sigma(\theta)\\
&= K \int \mu\{ y:  |\nu_{\theta,y}|\ge\tau \} \,d\sigma(\theta).
\end{align*}
From this and Fubini, we conclude that
\[
(\mu\times\sigma)\{ (y,\theta):  |\nu_{\theta,y}|\ge\tau \} \gtrsim \frac{1}{K},
\]
as desired.
\end{proof}

\subsection{Finding good vantage points}

We are now ready to apply Theorem \ref{thm:large-direction-set} to establish one of the main steps in the proof of Theorem \ref{thm:distance-set}. We begin with some notation.

Let $(m_j)_{j\in\N_0}$ be an increasing sequence such that $m_0=0$. Write $d_j=m_{j+1}-m_j$. Given $x$ such that $\mu(D_{m_j}(x))>0$, write
\[
\mu^{x,j}=\mu^{D_{m_j}(x),(d_j)}.
\]
That is, $\mu^{x,j}$ is the discretization at scale $d_j$ of the conditional measure on $D_{m_j}(x)$.
\begin{theorem} \label{thm:good-base-point}
Assume  $m_j + j\le  m_{j+1}$ for all $j\ge j_1$. Let $\mu,\nu\in\cP([0,1)^2)$ have disjoint supports and satisfy $\cE_{s'}(\mu),\cE_{s'}(\nu)<\infty$ for some $s'\in (1,2)$.  Given $\e>0$, $s\in [1,2)$ and $j_0\in\N$, define
\[
\Theta_x = \bigcap_{j=j_0}^\infty \Theta_x^{(j)}, \quad\text {where } \Theta_x^{(j)}=\left\{ \theta: \|\mu^{x,j}_\theta\|_2^2 \le 2^{\e d_j} \cE_s(\mu^{x,j})\right\}.
\]
If $j_0\ge j_1$ is fixed large enough (in terms of $\e, s',\cE_{s'}(\mu)$ and $\cE_{s'}(\nu)$), then
\[
(\mu \times \nu) \{ (x,y): \theta(x,y)\in \Theta_x \} \gtrsim 1,
\]
where the implied constant depends only on $\e, s',\cE_{s'}(\mu),\cE_{s'}(\nu),\dist(\supp(\mu),\supp(\nu))$. In particular, there exists $y\in \supp(\nu)$ such that
\[
\mu\{ x:\theta(x,y)\in\Theta_x \} \gtrsim 1.
\]
\end{theorem}
\begin{proof}
In the course of this particular proof, any constants implicit in the $\lesssim$ notation are allowed to depend on $\e, s',\cE_{s'}(\mu)$, $\cE_{s'}(\nu)$ and $\dist(\supp(\mu),\supp(\nu))$.

Let $j_0\ge j_1$. By Theorem \ref{thm:integral-sobolev-energy} and Markov's inequality, and since $d_j\ge j$ for $j\ge j_0$,
\begin{equation} \label{eq:small-set-of-exceptions}
\sigma(S^1\setminus \Theta_{x}) \le \sum_{j=j_0}^\infty \sigma(S^1\setminus \Theta_x^{(j)}) \lesssim  \sum_{j=j_0}^\infty 2^{-\e d_j} \lesssim 2^{-\e j_0},
\end{equation}
uniformly in $x$.  Note that $\cE_{s'}(\nu^{(m)})\le \cE_{s'}(\nu) \lesssim 1$. Hence, by Theorem \ref{thm:large-direction-set}, Fubini, and Eq. \eqref{eq:small-set-of-exceptions}, if $j_0$ is fixed large enough (in terms of $\cE_{s'}(\mu)$, $\cE_{s'}(\nu)$ and $\e$), then we can ensure that
\[
(\mu\times\sigma)\{ (x,\theta): \theta\in \Theta_{x} , |\nu^{(m)}_{\theta,x}|\ge \tau \}\gtrsim 1,
\]
for any $m\in\N$. Hence, there are Borel sets $G_m$ with $\mu(G_m)\gtrsim 1$, such that if $x\in G_m$, then
\begin{equation} \label{eq:large-measure-conditional-nu-large}
\sigma\{ \theta\in \Theta_{x} : |\nu^{(m)}_{\theta,x}|\ge \tau\} \gtrsim 1.
\end{equation}
Note that, since $\nu^{(m)}$ is absolutely continuous with $\cD_m$-measurable density, then
\begin{equation} \label{eq:conditional-nu-abs-cont}
|\nu^{(m)}_{\theta,x}|= \int_{\ell(x,\theta)} \nu^{(m)} \,d\cH^1,  \quad\theta\in S^1\setminus \{(0,\pm 1), (\pm 1, 0) \},
\end{equation}
where $\ell(x,\theta)$ is the line through $x$ with direction $\theta$. Let $g_x(y)=\theta(x,y)$, considered as a smooth map from a neighborhood of $\supp(\nu)$ to $S^1$, and denote the corresponding Jacobian by $Jg_x(y)$. Note that $J g_x(y)\lesssim 1$ when $x,y$ range in neighborhoods of $\supp(\mu),\supp(\nu)$ respectively (this is the estimate that depends on the distance between the supports). Applying the coarea formula (see e.g. \cite[Theorem 3.11]{EvansGariepy15}) to this map and the Borel function $y\mapsto \nu^{(m)}(y) \mathbf{1}_{\{\theta(x,y)\in\Theta_x\}}$, still for a fixed value $x\in G_m$, we get
\begin{align*}
\nu^{(m)}\{ y: \theta(x,y)\in \Theta_x \} &\gtrsim \int \nu^{(m)}(y)\mathbf{1}_{\{\theta(x,y)\in\Theta_x\}} J g_x(y) \,d\cH^2(y)\\
 &= \iint_{g_x^{-1}(\theta)}  \nu^{(m)}(y)\mathbf{1}_{\{\theta(x,y)\in\Theta_x\}} \,d\cH^1(y) \,d\sigma(\theta)\\
 &= \int_{\Theta_x} \int_{\ell(x,\theta)} \nu^{(m)}(y)\,d\cH^1(y) \,d\sigma(\theta)\\
 &= \int_{\Theta_x} |\nu_{\theta,y}^{(m)}|\,d\sigma(\theta)\\
 &\gtrsim 1,
\end{align*}
where we used \eqref{eq:large-measure-conditional-nu-large} and \eqref{eq:conditional-nu-abs-cont} in the last two lines. Recalling that $\mu(G_m)\gtrsim 1$, we deduce that
\[
\mu \times \nu^{(m)} \{ (x,y): \theta(x,y)\in \wt{\Theta}^{(j)}_x \} \gtrsim 1,\quad\text{where }\wt{\Theta}^{(j)}_x = \bigcap_{i=j_0}^j \Theta^{(i)}_x \supset \Theta_x.
\]
(Here we use that the supports of $\mu$ and $\nu^{(m)}$ are uniformly separated for $m$ large.) Endow $[0,1)^2\times [0,1)^2$ with the product dyadic metric. Note that $\mu\times\nu^{(m)}\to \mu\times\nu$ weakly in this topology. Moreover, the set $\{ (x,y): \theta(x,y)\in \wt{\Theta}^{(j)}_x \}$ is compact. Letting first $m\to\infty$ and then $j\to\infty$, we conclude that
\begin{align*}
\mu \times \nu \{ (x,y): \theta(x,y)\in \Theta_x \} &= \lim_{j\to\infty} (\mu\times\nu) \{ (x,y): \theta(x,y)\in \wt{\Theta}_x^{(j)}\}\\
&\ge \limsup_{j\to\infty}\left( \limsup_{m\to\infty} (\mu \times \nu^{(m)}) \{ (x,y): \theta(x,y)\in \wt{\Theta}_x^{(j)} \} \right)\\
&\gtrsim 1,
\end{align*}
as we wanted to show.
\end{proof}

\section{Proof of Theorem \ref{thm:distance-set}}
\label{sec:proof}

The core of the proof of Theorem \ref{thm:distance-set} is contained in the following statement.
\begin{prop} \label{prop:distance-set-technical}
Given $0<t<1<s$, if $\e>0$ is sufficiently small in terms of $t$, then there exists $\delta=\delta(\e)>0$  such that the following holds.

Suppose $\mu\in\cP([0,1)^2)$ satisfies $\ubdim(\supp(\mu)) \le s+\delta$ and
\begin{equation} \label{eq:almost-exact-dimensional}
\mu(Q) \le C 2^{-(m-\delta)s}
\end{equation}
for all $m$ and all $Q\in\cD_m$, and some $C>1$. Assume also that $B\subset [0,1)^2$ is closed and disjoint from $\supp(\mu)$ with $\hdim(B)>1$. Then there are $y\in B$ and a set $A_1$ with $\mu(A_1)>0$ such that,  if $m$ has the form $\lfloor (1+\e)^k\rfloor$ for some sufficiently large $k\in\N$ (depending on $\e, t,s,\mu$, $B$), and $A_2\subset A_1$ satisfies $\mu(A_2)\ge k^{-2}\mu(A_1)$, then
\[
H_m(\Delta_y \mu_{A_2}) \ge t.
\]
\end{prop}

We first show how to deduce Theorem \ref{thm:distance-set} from this proposition.
\begin{proof}[Proof of Theorem \ref{thm:distance-set}  (Assuming Proposition \ref{prop:distance-set-technical})]
To begin, note that it is enough to show that, given $t<1$,
\begin{equation} \label{eq:small-set-of-exceptional-pins}
\hdim(\{x:\hdim(\Delta_x A)< t\}) \le 1.
\end{equation}
Indeed, the claim with $t=1$ will then follow by taking a sequence $t_j\uparrow 1$. Fix, then, $0<t<1$ for the rest of the proof.

By replacing $A$ with a compact subset of dimension $\ge s-\delta/2$, it is enough to prove that if $A$ is compact and satisfies
\[
s-\delta/2 \le \hdim(A) \le \pdim(A) \le s,
\]
where $\delta$ is chosen small enough (in terms of $\e$, hence ultimately in terms of $t$), and $B\subset\R^2$  has Hausdorff dimension $>1$, then there is $y\in B$ such that
\begin{equation} \label{eq:claim-large-pinned-dim}
\hdim(\Delta_y A)\ge t.
\end{equation}
Since the map $y\to\hdim(\Delta_y A)$ can be checked to be Borel (for compact $A$), we may  also assume that $B$ is compact.

Using the scale and translation invariance of the problem, we can further assume that $A, B\subset [0,1)^2$. Furthermore, by picking a countable cover $B(x_i,r_i)$ of $B$ such that
\[
\hdim(A\setminus B(x_i,2r_i))>s- 2\delta/3,
\]
we can additionally assume that $A$ and $B$ are disjoint (now with $\hdim(A)>s-2\delta/3$).

By Frostman's Lemma (see \cite[Theorem 8.8]{Mattila95}), there exist $C_1>0$  and a measure $\nu$ supported on $A$ (both depending on $A$ and $\delta$), such that $\nu(B(x,r)) \le C_1 r^{s-\delta}$ for all $x\in \supp(\nu)$. On other hand,  using the characterization of packing dimension in terms of upper box dimension (see \cite[Proposition 3.8]{Falconer14}), we get that
\[
s\ge \pdim(A) = \inf\left\{ \sup_i \ubdim(A_i): A \subset \bigcup_i A_i\right\},
\]
and since $\ubdim(A_i)=\ubdim(\overline{A}_i)$, it follows that there is a closed set $A_0$ such that $\nu(A_0)>0$ and $\ubdim(A_0)<s+\delta$. We deduce that $\mu=\nu_{A_0}$ satisfies the properties in the statement of the Proposition \ref{prop:distance-set-technical}.

Let, then, $y\in B$ be the point given by Proposition \ref{prop:distance-set-technical}, and let $A_1$ be the corresponding set. Assume $\cH^u(\Delta_y A_1)=0$ for some $u$; our goal is to show that $u$ must be large. Let $\{ I_j\}$ be a cover of $\Delta_y(A_1)$ by dyadic intervals with $\sum_j |I_j|^u \le 1$, where the largest interval is very small. By enlarging the $I_j$, we obtain new intervals (still denoted $I_j$) such that $|I_j|=2^{-\lfloor (1+\e)^k\rfloor}$ for some $k=k(j)$, where the smallest $k$ is arbitrarily large, and
\begin{equation} \label{eq:hausdorff-sum-small}
\sum_j |I_j|^{u(1+\e)}\le 1.
\end{equation}
Pigeonholing, we can find a fixed (large) value of $k$ such that
\[
\mu(A_2) > k^{-2}\mu(A_1), \quad\text{where } A_2=\{ x\in A_1: |x-y|\in I_j \text{ with } |I_j| = 2^{-\lfloor (1+\e)^k\rfloor} \}.
\]
Now Proposition \ref{prop:distance-set-technical} ensures that $H_m(\Delta_y \mu_{A_2}) \ge t$ with $m=2^{\lfloor (1+\e)^k\rfloor}$; in particular, using \eqref{eq:upper-bound-entropy-size-partition},
\[
|\{j: |I_j|= 2^{-m} \}| \ge 2^{mt}.
\]
Combining this with \eqref{eq:hausdorff-sum-small}, we conclude that
\[
2^{mt} 2^{-m(u(1+\e))} \le 1.
\]
That is,
\[
\hdim(\Delta_y A) \ge \hdim(\Delta_y A_2) \ge u \ge \frac{t}{1+\e}.
\]
This establishes the claim \eqref{eq:claim-large-pinned-dim} with $t/(1+\e)$ in place of $t$ which, since we are allowed to take $\e$ arbitrarily small, finishes the proof.
\end{proof}

\begin{proof}[Proof of Proposition \ref{prop:distance-set-technical}]
Fix $0<t<1<s<1$ and a small $\e>0$ (to be determined in the course of the proof). Let $\delta=\e^2$, where $\e$ is small enough that $s-2\delta>1$. By assumption, there exists a constant $C>0$ such that
\begin{equation} \label{eq:almost-exact-dimension}
\mu(Q) \le C 2^{-(s-\delta)m} \quad\text{for all } Q\in\cD_m,
\end{equation}
and $\supp(\mu)$ can be covered by $C 2^{(s+2\delta)m}$ squares in $\cD_m$, in both cases for all $m\in\N$. In particular, letting $s_1=s-2\delta$, we have $\cE_{s_1}(\mu)<\infty$. In the $\lesssim$ notation below, the implicit constants are allowed to depend on $s, \delta$ and $C$, but not on the scales $m, j$ or $k$. Note that
\begin{equation} \label{eq:light-squares-small-total-mass}
\sum \{ \mu(Q): Q\in\cD_m, \mu(Q) \le 2^{-(s_1+5\delta)m} \} \lesssim 2^{-\delta m}.
\end{equation}
Let us call a square $Q\in\cD_m$ \emph{heavy} if $\mu(Q) >2^{-(s_1+5\delta)m}$.

Define $m_j = \lfloor (1+\e)^j \rfloor$ for $j\ge 1$ and $m_0=0$. Note that $d_j=m_{j+1}-m_j$ satisfies $|d_j- \e m_j|\lesssim 1$ and $d_j \le m_j+1$ for all $j$, provided $\e<1/2$. Applying Lemma \ref{lem:multiscale-energy} (and the remark after it) to $\mu$ and this sequence, we obtain
\begin{align*}
1&\approx \cE_{s_1}(\mu) \approx \cE_{s_1}(\mu^{(m_k)})\\
&\approx \sum_{j=0}^{k-1} 2^{s_1 m_j } \sum_{Q\in\cD_{m_j}} \mu(Q)^2 \cE_{s_1}(\mu^{Q,(d_j)}) \\
&>\sum_{j=0}^{k-1} 2^{-5 \delta m_j} \sum_{Q\in\cD_{m_j} \text{ is heavy}} \mu(Q) \cE_{s_1}(\mu^{Q,(d_j)}).
\end{align*}
In particular, this shows that for any $j$,
\[
 \sum_{Q\in\cD_{m_j} \text{ is heavy}} \mu(Q) \cE_{s_1}(\mu^{Q,(d_j)}) \lesssim 2^{5\delta m_j},
\]
which, by Markov's inequality and \eqref{eq:light-squares-small-total-mass}, implies that
\begin{equation} \label{eq:local-energy-large}
\sum \{ \mu(Q): Q\in\cD_{m_j}, \cE_{s_1}(\mu^{Q,(d_j)}) \ge 2^{6 \delta m_j} \} \lesssim 2^{-\delta m_j}.
\end{equation}

Since $\hdim(B)>1$, by Frostman's Lemma there exist $s_2\in (1,s-2\delta)$ and $\nu\in\cP([0,1)^2)$, supported on $B$, such that $\cE_{s_2}(\nu)<\infty$. Applying Theorem \ref{thm:good-base-point} (with $s'=s_2$, $s=s_1$) we obtain $j_0\in\N$ and a point $y\in B$ such that $\mu(A_1)\gtrsim 1$, where
\begin{equation} \label{eq:def-A1}
A_1=\left\{ x\in\supp(\mu):  \|\mu^{x,j}_{\theta(x,y)}\|_2^2 \le 2^{\e d_j} \cE_{s_1}(\mu^{x,j}) \text{ for all } j\ge j_0 \right\},
\end{equation}
where we recall that  $\mu^{x,j}= \mu^{D_{m_j}(x),(d_j)}$.

Now suppose $A_2\subset A_1$ satisfies $\mu(A_2)> k^{-2}\mu(A_1)\gtrsim k^{-2}$ for some $k\gg j_0$. Write $\wt{\mu}=\mu_{A_2}$ for simplicity. Since $A_2\subset \supp(\mu)$, there are $\lesssim 2^{m(s_1+4\delta)}$ squares $Q\in\cD_m$ with positive $\wt{\mu}$-mass, so we get that, for any $j$,
\[
\sum \{ \wt{\mu}(Q):Q\in\cD_{m_j}, \wt{\mu}(Q) \le 2^{-(s_1+5\delta)m_j}\} \lesssim 2^{-\delta m_j}.
\]
Using \eqref{eq:almost-exact-dimension}, we deduce that
\[
\sum \{ \wt{\mu}(Q):Q\in\cD_{m_j}, \wt{\mu}(Q) \le 2^{- 4 \delta m_j}\mu(Q) \} \lesssim 2^{- \delta m_j}.
\]
On the other hand, from \eqref{eq:local-energy-large} and $\mu(A_2)\gtrsim k^{-2}$ we also get
\[
\sum \left\{ \wt{\mu}(Q):Q\in\cD_{m_j}, \cE_{s_1}\left(\mu^{Q,(d_j)}\right) \ge 2^{6  \delta m_j} \right\} \lesssim k^2 2^{- \delta m_j}.
\]
Call a square $Q\in\cD_{m_j}$ \emph{good} if $\wt{\mu}(Q) \ge 2^{- 4 \delta m_j}\mu(Q)>0$ and $\cE_{s_1}\left(\mu^{Q,(d_j)}\right) \le 2^{6  \delta m_j}$, and call it \emph{bad} otherwise. We deduce from the last two displayed equations that, if $k/2 \le j\le k$ and $k$ is large enough, then
\begin{equation} \label{eq:mass-bad-squares-small}
\sum \{ \wt{\mu}(Q): Q\in\cD_{m_j} \text{ is bad } \} \le 2^{-\delta m_j/2}.
\end{equation}
Now, note that for any Borel set $S\subset\R^2$,
\[
\wt{\mu}_Q(S) = \frac{\mu(S\cap Q\cap A_2)}{\mu(Q\cap A_2)} \le \frac{\mu(Q)}{\mu(Q\cap A_2)}\mu_Q(S) \lesssim \left(k^2 \frac{\mu(Q)}{\wt{\mu}(Q)}\right)\,\mu_Q(S).
\]
This pointwise domination is preserved under push-forwards and discretizations. In particular,
\[
\wt{\mu}^{Q,(m_j)}_{\theta(x,y)}(z)  \lesssim  \left( k^2 \frac{\mu(Q)}{\wt{\mu}(Q)}\right)  \,\mu^{Q,(m_j)}_{\theta(x,y)}(z)
\]
for all $z$. We deduce that if $Q\in\cD_{m_j}$ is a good square (and always assuming $k/2\le j\le k$), then
\[
\left\|\wt{\mu}^{Q,(m_j)}_{\theta(x,y)}\right\|_2^2  \lesssim 2^{9 \delta m_j} \left\|\mu^{Q,(m_j)}_{\theta(x,y)}\right\|_2^2 .
\]
Using this and taking into account \eqref{eq:def-A1} and the fact that $A_2\subset A_1$, we deduce that
if $x\in Q\cap A_2$, with $Q\in\cD_{m_j}$ a good square, and $k/2\le j\le k$, then
\begin{align*}
\left\|\wt{\mu}^{x,j}_{\theta(x,y)}\right\|_2^2 &\lesssim  2^{9 \delta m_j} \left\|\mu^{x,j}_{\theta(x,y)}\right\|_2^2 \\
&\le 2^{9 \delta m_j} 2^{\e d_j} 2^{6  \delta m_j} \\
&\approx 2^{16\e d_j},
\end{align*}
using that $\delta=\e^2$ and $|d_j-\e m_j|\lesssim 1$ in the last line.

Applying Lemmas \ref{lem:L2-to-entropy} and \ref{lem:discretize-project}, we see that if $Q\in\cD_{m_j}$ is a good square and $k/2\le j<k$, then there is $x_Q\in Q$ such that
\begin{equation} \label{eq:good-square-large-entropy}
H_{d_j}\left( \wt{\mu}^Q_{\theta(x_Q,y)} \right) \ge 1 - 16\e - C/d_j \ge 1-17\e,
\end{equation}
taking $k$ even larger if needed.

To conclude, and making $\e$ smaller and $k$ even larger in terms of $\e$, we invoke Proposition \ref{prop:entropy-of-pinned-dist-measures}:
\begin{align*}
H(\Delta_y\wt{\mu},\cD_{m_k}) &\ge -C k +\sum_{j=0}^{k-1} \sum_{Q\in\cD_{m_j}} \wt{\mu}(Q) H\left(\wt{\mu}^Q_{\theta(y,x_Q)} ,\cD_{d_j}\right)\\
&\ge -\e m_k + \sum_{j=k/2}^{k-1} \sum_{Q\in\cD_{m_j}\text{ good}} \wt{\mu}(Q) H\left(\wt{\mu}^Q_{\theta(y,x_Q)} ,\cD_{d_j}\right)\\
&\ge -\e m_k + \sum_{j=k/2}^{k-1} (1-2^{-\delta m_j/2}) (1-17\e) d_j\\
&\ge -\e m_k + (1-\e)(1-17\e)(m_k-m_{\lfloor k/2\rfloor})\\
&\ge (1-20\e)m_k,
\end{align*}
using \eqref{eq:mass-bad-squares-small} and \eqref{eq:good-square-large-entropy} in the third line. This establishes the proposition if we take $\e\le(1-t)/20$.
\end{proof}


\end{document}